\documentclass[]{interact}

\usepackage{epstopdf}
\usepackage[caption=false]{subfig}

\usepackage[numbers,sort&compress]{natbib}
\bibpunct[, ]{[}{]}{,}{n}{,}{,}
\makeatletter
\def\NAT@def@citea{\def\@citea{\NAT@separator}}
\makeatother

\theoremstyle{plain}
\newtheorem{theorem}{Theorem}[section]
\newtheorem{lemma}[theorem]{Lemma}

\theoremstyle{definition}

\newtheorem{ex}[theorem]{Example}

\theoremstyle{remark}

\def\Bar{\overline}

\def\B{I\!\!B}

\def\gph{\mbox{\rm gph}\,}

\begin{document}


\title{Stability of Nonhomogeneous Split Equality and Split Feasibility Problems with Possibly Nonconvex Constraint Sets}

\author{
\name{Vu Thi Huong\textsuperscript{a,b}\thanks{Emails: huong.vu@zib.de; vthuong@math.ac.vn}, Hong-Kun Xu\textsuperscript{c,d}\thanks{Email: xuhk@hdu.edu.cn. Corresponding author}, and Nguyen Dong Yen\textsuperscript{b}\thanks{Email: ndyen@math.ac.vn}}
\affil{
\textsuperscript{a}{Digital Data and Information for Society, Science, and Culture\\ Zuse Institute Berlin, 14195 Berlin, Germany}\\
 \textsuperscript{b}{Institute of Mathematics, Vietnam Academy of Science and Technology, 10072 Hanoi, Vietnam}\\
 \textsuperscript{c}{School of  Science, Hangzhou Dianzi University, 310018 Hangzhou, China}\\
 \textsuperscript{d}{College of Mathematics and Information Science, Henan Normal University, 453007 Xinxiang,  China}
 }}

\maketitle

\begin{abstract}
By applying some techniques of set-valued and variational analysis, we study solution stability of nonhomogeneous split equality problems and nonhomogeneous split feasibility problems, where the constraint sets need not be convex. Necessary and sufficient conditions for the Lipschitz-likeness of the solution maps of the problems are given and illustrated by concrete examples. The obtained results complement those given in~[Huong VT, Xu HK, Yen ND. {Stability analysis of split equality and split feasibility problems}. arXiv:2410.16856.], where classical split equality problems and split feasibility problems have been considered.  
\end{abstract}

\begin{keywords}
Nonhomogeneous split equality problem, nonhomogeneous split feasibility problem, Lipschitz-likeness, limiting normal cone, limiting coderivative, parametric generalized equation, adjoint operator.

\medskip
 \noindent {\bf 2020 Mathematics Subject Classification:}  49J53, 49K40, 65K10, 90C25, 90C31.
\end{keywords}

\section{Introduction} \setcounter{equation}{0}

To model phase retrieval and other image restoration problems in signal processing, Censor and Elfving~\cite[Section~6]{Censor_Elfving_94} introduced  the \textit{split feasibility problem} (SFP). Since the problem is also useful in signal processing/imaging reconstruction, medical treatment of intensity-modulated radiation therapy, gene regulatory network inference, etc., it has been studied by many authors; see, e.g.,~\cite{HuongXuYen22} for some related references. Apart from (SFP), another feasibility problem called the
\textit{split equality problem} (SEP), which was proposed by  Moudafi and his coauthors~\cite{Moudafi13,Moudafi14,Moudafi_Shemas13,Byrne_Moudafi_17}, can be used in decomposition methods for PDEs, game theory, and intensity-modulated radiation therapy. Relationships between (SEP) and (SFP) can be seen in~\cite[Section~3]{Xu_Celgieski21}.

The study of (SEP) and (SFP) is mainly focused on solution algorithms for (SFP). The interested reader is referred to~\cite{HuongXuYen22} for several references and comments. 

In~\cite{HuongXuYen22}, we have shown that solution stability of the above feasibility problems can be successfully investigated by set-valued and variational analysis techniques. The main idea is to 
transform the problems into suitable parametric generalized equations and use the machinery of generalized differentiation from~\cite{B-M06}, as well as a fundamental result of Mordukhovich~\cite{Mor04}. The results obtained in~\cite{HuongXuYen22} include necessary and sufficient conditions for the Lipschitz-like property of the solution map in question. A discussion on the importance of the latter property in and its relations to other stability property of set-valued maps can be found in~\cite{HuongXuYen22}. 

Introducing a canonical perturbation, usually termed the right-hand-side perturbation, to the constraint system of (SEP) (resp., of (SFP)), we get a \textit{nonhomogeneous split equality problem} (resp.,~\textit{nonhomogeneous split feasibility problem}). The first model has been considered by Reich and Tuyen~\cite{Reich_Tuyen_21}. It is worth emphasizing that the consideration of nonhomogeneous systems is a traditional issue in mathematics. For instance, in parallel to homogeneous systems of linear equations  (resp., homogeneous systems of linear differential equations), one also considers nonhomogeneous systems of linear equations  (resp., nonhomogeneous systems of linear differential equations).

Our aim in this paper is to characterize the Lipschitz-like property of the solution map of the nonhomogeneous split equality problem and of the nonhomogeneous split feasibility problem where \textit{the constraint sets are not required to be convex}. The aim will be achieved by developing some proof schemes used in~\cite{HuongXuYen22}. Due to the appearance of the canonical perturbation (the right-hand-side perturbation), the obtained results cannot be derived from those given in~\cite{HuongXuYen22}. Conversely, the characterizations of the Lipschitz-like property of the solution maps in the paper do not imply the ones of~\cite{HuongXuYen22}. 

The paper organization is as follows. Section~\ref{sect_Preliminaries} presents the notions and auxiliary results needed in the sequel. Nonhomogeneous split equality problems is studied in Section~\ref{sect_NSEP}. Section~\ref{sect_NSFP} is devoted to nonhomogeneous split feasibility problems. Two illustrative examples for the obtained solution stability results are presented in Section~\ref{sect_Examples}. The last section gives several concluding remarks.      

Throughout the paper, if $M\in\mathbb R^{p\times q}$ is a matrix, then $M^{\rm T}$ stands for the transpose of $M$. The kernel is defined by ${\rm ker}\, M=\big\{ x\in\mathbb R^p\,:\, Mx=0\big\}$. The inverse of a set $\Omega\subset\mathbb R^q$ via the operator $M:\mathbb R^p\to\mathbb R^q$  is defined by $M^{-1}(\Omega)=\big\{ x\in\mathbb R^p\,:\, Mx\in\Omega\big\}$. The space of the linear operators from $\mathbb R^p$ to $\mathbb R^q$ is denoted by $L\big(\mathbb R^p,\mathbb R^q\big)$. The symbol ${\rm int}\,\Omega$ signifies the topological interior of $\Omega$. The open ball (resp., the closed ball) with center $\bar y\in\mathbb R^q$ and radius $\rho>0$ are denoted by ${\B}(\bar y,\rho)$ (resp., $\Bar {\B}(\bar y,\rho)$). The set of nonnegative real numbers is abbreviated to~$\mathbb R_+$.

\section{Preliminaries}\label{sect_Preliminaries} \setcounter{equation}{0}

Some basic concepts and tools from set-valued and variational analysis~\cite{B-M06,Mordukhovich_2018}, which will be needed the sequel, are recalled in this section.

\medskip
Let  $F:\mathbb R^n\rightrightarrows  \mathbb R^m$ be a set-valued map. The \textit{graph} of $F$ is the set $$\gph F:=\big\{(x,y)\in\mathbb R^n\times\mathbb R^m\,:\, y\in F(x)\big\}.$$ 
We say that $F$ has \textit{closed graph} if $\gph F$ is closed in the product space $\mathbb R^n\times\mathbb R^m$, where the norm is given by $\|(x,y)\|:=\|x\|+\|y\|$ for all $(x,y)\in\mathbb R^n\times\mathbb R^m$. For any $(\bar x,\bar y)\in \gph F$, one says that $F$ is {\it locally closed} around $(\bar x,\bar y)$ if there exits $\rho>0$ such that $\big(\gph F\big)\cap \Bar {\B}((\bar x, \bar y),\rho)$ is closed in $\mathbb R^n\times\mathbb R^m$.  If $F$ has closed graph, then it is locally closed around any point in its graph. 

\medskip
Let $\Omega$ be a nonempty subset of $\mathbb{R}^n$ and $\bar x \in \Omega$. The {\it regular normal cone} (or \textit{Fr\'echet normal cone}) to $\Omega$ at $\bar x$ is defined by
\begin{align*}
	\widehat N(\bar x; \Omega)=\Big\{ x'\in \mathbb{R}^n\; :\; \limsup\limits_{x \xrightarrow{\Omega}\bar x} \dfrac{\langle x', x-\bar x \rangle}{\|x-\bar x\|} \leq 0 \Big\},
\end{align*}
where $x \xrightarrow{\Omega} \bar x$ means that $x \rightarrow \bar x$ and $ x\in \Omega$. The {\it limiting normal cone} (or {\it Mordukhovich normal cone}) to $\Omega$ at $\bar x$ is given by
\begin{eqnarray*}\begin{array}{rl}N(\bar x;\Omega)&:=\big\{x'\in \mathbb{R}^n\, :\,  \exists \mbox{ sequences } x_k\to \bar x,\ x_k'\rightarrow x'\\
		& \qquad \qquad \qquad \quad \ \mbox {with } x_k'\in \widehat
		N(x_k;\Omega)\, \mbox{ for all }\, k=1,2,\dots
		\big\}.\end{array}\end{eqnarray*} 
We put $\widehat N(\bar x; \Omega)=N(\bar x; \Omega) =\emptyset$ if $\bar x \not\in \Omega$. One has $\widehat N(\bar x; \Omega) \subset N(\bar x; \Omega)$ for all $\Omega \subset \mathbb{R}^n$ and $\bar x \in \Omega$. If the reverse inclusion holds, one says ~\cite[Def.~1.4]{B-M06} that~$\Omega$ is \textit{normally regular} at~$\bar x$. It is well known~\cite[Prop.~1.5]{B-M06} that if  $\Omega$ is convex, then both regular normal cone and limiting normal cone to $\Omega$ at $\bar x$  reduce to the \textit{normal cone in the sense of convex analysis}, that is,
\begin{align*}
	\widehat{N}( \bar x; \Omega)= {N}( \bar x; \Omega)=\big\{x'\in \mathbb{R}^n \,:\, \langle x', x-\bar x \rangle \le 0\ \,\forall x\in\Omega\big\}.
\end{align*}
Thus, a convex set is normally regular at any point belonging to it. There exist many nonconvex sets which are normally regular. 

\begin{ex}\label{ex_a} {\rm The set $\Omega_1:=\big\{x=(x_1, x_2)\in \mathbb{R}^2\, :\, x_1\leq x_2^2 \big\}$ is normally regular at any point $\bar x=(\bar x_1, \bar x_2)\in\Omega_1$. Indeed, setting $f(x)=x_1-x_2^2$ for all $x=(x_1, x_2)\in\mathbb{R}^2$ and $\Theta=-\mathbb R_+$, one has $\Omega_1=f^{-1}(\Theta):=\{x\in\mathbb R^2\,:\, f(x)\in\Theta\}$. Applying~\cite[Theorem~1.19]{B-M06}, we can assert that $\Omega_1$ is normally regular at~$\bar x$. Indeed, setting $f(x)=x_1-x_2^2$ for all $x=(x_1, x_2)\in\mathbb{R}^2$ and $\Theta=-\mathbb R_+$, one has $\Omega_1=f^{-1}(\Theta):=\{x\in\mathbb R^2\,:\, f(x)\in\Theta\}$. Since $f$ is continuously differentiable and $\nabla f(\bar x)\neq 0$, the normal regularity of $\Omega_1$  at~$\bar x$ is assured by~\cite[Theorem~1.19]{B-M06}.}
\end{ex}

\begin{ex}\label{ex_b} {\rm The set 
		$\Omega_2:=\big\{x=(x_1, x_2)\in \mathbb{R}^2\, :\, \gamma _1\leq x_1^2+x_2^2 \leq \gamma_2\big\}$ with $0<\gamma_1<\gamma_2$ is normally regular at any point $\bar x=(\bar x_1, \bar x_2)\in\Omega_2$. Indeed, with $g(x):=x_1^2+x_2^2$ for all $x=(x_1, x_2)\in\mathbb{R}^2$ and $\Theta:=[\gamma_1,\gamma_2]$, one has $\Omega_1=g^{-1}(\Theta):=\{x\in\mathbb R^2\,:\, g(x)\in\Theta\}$. As $g$ is continuously differentiable and $\nabla g(\bar x)\neq 0$, applying~\cite[Theorem~1.19]{B-M06} shows that $\Omega_2$ is normally regular at~$\bar x$. Note that $\Omega_2$ is nonconvex, compact, and path connected.}
\end{ex}

Let $(\bar x,\bar y)$ belong to the graph of a set-valued map $F:\mathbb R^n\rightrightarrows  \mathbb R^m$. The set-valued map $D^*F(\bar x,\bar y):\mathbb R^m\rightrightarrows \mathbb R^n$ with
\begin{eqnarray*}
	D^*F(\bar x,\bar y)(y'):=\big\{x'\in \mathbb R^n\, :\, (x',-y')\in N((\bar x,\bar y);\mbox{gph}\, F)\big\}
\end{eqnarray*}
is called the {\it limiting coderivative} (or {\it Mordukhovich coderivative}) of the set-valued map $F$ at $(\bar x,\bar y)$.  When $F$ is single-valued and $\bar y=F(\bar x)$, we write $D^*F(\bar x)$ for $D^*F(\bar x,\bar y)$. If $F:\mathbb R^n\to \mathbb R^m$ is \textit{strictly differentiable} \cite[p.~19]{B-M06} at $\bar x$ (in particular, if $F$ is continuously Fr\'echet differentiable in a neighborhood of $\bar x$) with the derivative $\nabla F(\bar x)$ then
\begin{eqnarray*}
	D^*F(\bar x)(y')=\big\{\nabla F(\bar x)^*y'\big\}, \quad \forall	y'\in\mathbb R^m.
\end{eqnarray*} 
Here, the {\it adjoint operator} $\nabla F(\bar x)^*$ of $\nabla F(\bar x)$ is defined by setting $$\langle\nabla F(\bar x)^*y',x\rangle=\langle y',\nabla F(\bar x)x\rangle$$ for every $x\in\mathbb R^n$ (see \cite[Theorem 1.38]{B-M06} for more details). In accordance with the definition given at~\cite[p.~351]{Mor04}, we say that $F$ is \textit{graphically regular} at $(\bar x,\bar y) \in  \mbox{gph}\, F$ if $\mbox{gph}\,F$ is normally regular at $(\bar x,\bar y).$

\medskip
One says that $F:\mathbb R^n\rightrightarrows  \mathbb R^m$ is \textit{Lipschitz-like} (\textit{pseudo-Lipschitz}, or has the \textit{Aubin property}; see \cite{Aubin84}) at $(\bar x, \bar y) \in \gph F$ if there exist neighborhoods $U$ of $\bar x$, $V$ of $\bar y$, and constant $\ell > 0$ such that
\begin{align*}
	F(x') \cap V \subset F (x) + \ell \| x' - x\| \Bar {\B}_{\mathbb{R}^m}, \quad \forall x',\, x \in U.
\end{align*} 

Let $X, Y, Z$ be finite-dimensional spaces.  Consider a \textit{parametric generalized equation}
\begin{equation}\label{para_general_eq}
	0\in f(x, y)+G(x, y)
\end{equation}
with the decision variable $y$ and the parameter $x$, where $f:X\times Y\rightarrow Z$ is a single-valued map while $G:X\times Y\rightrightarrows Z$ is a set-valued map. The \textit{solution map} to \eqref{para_general_eq} is the set-valued map $S: X \rightrightarrows Y$ given by
\begin{equation}\label{solution_map}
	S(x):=\left\{y\in Y \;:\; 0\in f(x, y)+G(x, y)\right\}, \quad x\in X.
\end{equation}

\medskip
The next theorem, which is a special case of~\cite[Theorem 4.2(ii)]{Mor04} states a necessary and sufficient condition for Lipschitz-like property of the solution map \eqref{solution_map}.
\begin{theorem}{\rm (See~\cite[Theorem 4.2(ii)]{Mor04})}\label{Thm. 4.2-Mor-2004-JOGO}
	Let $X, Y, Z$ be finite-dimensional spaces and let $(\bar x, \bar y)$ satisfy \eqref{para_general_eq}. Suppose that $f$ is strictly differentiable at $(\bar x, \bar y)$, $G$ is locally closed around  $(\bar x, \bar y,\bar z)$ with $\bar z:=-f(\bar x, \bar y)$ and, moreover, $G$ is graphically regular at $(\bar x, \bar y,\bar z)$. If
	\begin{equation}\label{Lipschitz-like condition}
		\big [(x', 0)\in \nabla f(\bar x, \bar y)^*(z')+D^*G(\bar x, \bar y, \bar z)(z') \big]\Longrightarrow [x'=0,\, z'=0],
	\end{equation} then $S$ is Lipschitz-like at $(\bar x, \bar y)$. Conversely, if $S$ is Lipschitz-like at $(\bar x, \bar y)$ and if the regularity condition
	\begin{equation}\label{regularity condition}
		\big [(0,0)\in \nabla f(\bar x, \bar y)^*(z')+D^*G(\bar x, \bar y, \bar z)(z') \big]\Longrightarrow z'=0\end{equation}
	is satisfied, then one must have~\eqref{Lipschitz-like condition}.
	
\end{theorem}

\section{Nonhomogeneous Split Equality Problems}\label{sect_NSEP} \setcounter{equation}{0}

In this section, we propose the notion of nonhomogeneous split equality problem with possibly nonconvex constraint sets and study the solution map under the perturbations of the two matrices and the vector in the data tube.

\medskip
Let ${\mathcal C}\subset \mathbb{R}^n$, ${\mathcal Q}\subset \mathbb{R}^m$ be nonempty, closed sets and let $A\in \mathbb{R}^{l\times n}$, $B \in  \mathbb{R}^{l \times m}$ and $c \in \mathbb R^l$ be given matrices and vector. The \textit{nonhomogeneous split equality problem} (NSEP) defined by the data tube $(A, B, c, \mathcal C, \mathcal Q)$ aims at seeking all the pairs $(x, y) \in \mathcal C \times \mathcal Q$ such that $Ax - By = c$. 

\medskip
Especially, if the sets $\mathcal C, \mathcal Q$ are convex and $c = 0$, then (NSEP) becomes the widely-known \textit{split equality problem} (SEP) considered; see~\cite{HuongXuYen22} for a large list of related references. 

\medskip
Assuming that both sets $\mathcal C$ and $\mathcal Q$ are convex, Reich and Tuyen~\cite{Reich_Tuyen_21} considered several iterative methods for solving (NSEP).

\medskip
We are interested in the solution stability of (NSEP) w.r.t. the changes of the parameters $A, B, c$. The solution map $\Phi : \mathbb{R}^{l\times n} \times \mathbb{R}^{l \times m} \times \mathbb R^l \rightrightarrows \mathbb{R}^{n} \times \mathbb{R}^{m}$ of (NSEP) is defined by 
\begin{equation}\label{phi}
	\Phi (A, B,c):= \{(x, y)  \in \mathcal C \times \mathcal Q \, : \, Ax-By=c\}, \quad (A, B,c) \in \mathbb{R}^{l\times n} \times \mathbb{R}^{l \times m} \times \mathbb R^l.
\end{equation}

Our results on sufficient and necessary conditions for the Lipschitz-like property of $\Phi $ can be stated as follows.

\begin{theorem}\label{Lipschitz-like_phi_thm}
	Let $(\bar A, \bar B, \bar c) \in \mathbb{R}^{l\times n} \times \mathbb{R}^{l\times m} \times \mathbb R^l$ and $(\bar x, \bar y) \in \Phi (\bar A,\bar B, \bar c)$. Suppose that $\mathcal C$ is normally regular at $\bar x$ and $\mathcal Q$ is normally regular at $\bar y$. If the regularity condition
	\begin{equation}\label{Lipschitz-like_S_con}
		\big(\bar A^{\rm T}\big)^{-1}\big(-N(\bar x; \mathcal C)\big) \cap \big(\bar B^{\rm T}\big)^{-1}\big(N(\bar y; \mathcal Q)\big)  = \{0\}.
	\end{equation} is satisfied, then the solution map $\Phi$ of {\rm (NSEP)} is Lipschitz-like at $\big((\bar A, \bar B, \bar c),(\bar x, \bar y)\big)$. Conversely, if $\Phi$ is Lipschitz-like at $\big((\bar A, \bar B, \bar c),(\bar x, \bar y)\big)$, then the regularity condition~\eqref{Lipschitz-like_S_con} holds, provided that $(\bar x, \bar y) \neq (0, 0)$.
\end{theorem}

To prove Theorem~\ref{Lipschitz-like_phi_thm}, we will transform (NSEP) to a generalized equation of a special type and apply Theorem~\ref{Thm. 4.2-Mor-2004-JOGO}. This approach is similar to the one for obtaining Theorem~3.1 in~\cite{HuongXuYen22}, where the Lipschitz-likeness of (homogeneous) split equality problems was investigated. 

\medskip
Let $W: = \mathbb{R}^{l\times n} \times \mathbb{R}^{l\times m} \times \mathbb R^l$, $U:= \mathbb{R}^n \times \mathbb{R}^m$, and $V:= \mathbb{R}^n \times \mathbb{R}^m \times \mathbb{R}^l$. Consider the function $f_1: W\times U \to V$ given by   
\begin{equation}\label{f1}
	f_1(w, u):= (-x, -y, Ax - By-c), \quad w=(A, B, c) \in  W,\ u=(x, y) \in U
\end{equation}
and the set-valued map $G_1: W\times U \rightrightarrows V$ with 
\begin{equation}\label{G1}
	G_1(w, u):= \mathcal C \times \mathcal Q \times \{0_{\mathbb{R}^l}\}, \quad w=(A, B, c) \in  W, \ u=(x, y) \in U.
\end{equation}
Then, the solution map $(A, B, c) \mapsto \Phi(A, B, c)$ of (NESP) can be interpreted as the solution map $w\mapsto \Phi(w)$ of the parametric generalized equation 
$0 \in f_1(w, u) + G_1(w, u)$ where, in agreement with~\eqref{phi},  
\begin{equation}\label{phi_solution map}
	\Phi(w) = \left\{u \in U \; : \; 0 \in f_1(w, u) + G_1(w, u)\right\}, \quad w \in W.
\end{equation}

First, let us compute the derivative of $f_1$. 

\begin{lemma} \label{lemma_f1}
	The function $f_1: W\times U \to V$ in~\eqref{f1} is strictly differentiable at any point $(\bar w, \bar u) \in W \times U$ with $\bar w = (\bar A, \bar B, \bar c)$, $\bar u =(\bar x, \bar y)$. The derivative $\nabla f_1(\bar w, \bar u) : W \times U \to V$  of $f_1$ at $(\bar w, \bar u)$ is given by the formula
	\begin{equation}\label{deri_f1}
		\nabla f_1(\bar w, \bar u)(w, u) = (-x, -y, \bar A x - \bar B y + A \bar x - B \bar y- c) 
	\end{equation} 
	for $w = (A, B, c) \in W,\ u =(x, y) \in U$. Moreover, the operator $\nabla f_1(\bar w, \bar u) : W \times U \to V$ is surjective and thus its adjoint operator $\nabla f_1(\bar w, \bar u)^* : V \to W \times U$ is injective whenever $\bar u \neq (0, 0)$.
\end{lemma}

\begin{proof}
	Fix any $(\bar w, \bar u) \in W \times U$ with $\bar w = (\bar A, \bar B, c)$, $\bar u =(\bar x, \bar y)$. Clearly, the formula 
	\begin{equation*}\label{T_1}
		T_1(\bar w, \bar u)(w, u) := (-x, -y, \bar A x - \bar B y + A \bar x - B \bar y -c),\ \; w = (A, B, c) \in W,\ u =(x, y) \in U 
	\end{equation*} defines a linear operator $T_1(\bar w, \bar u): W \times U \to V$. Setting 
	$$L_1 = \displaystyle\lim_{(w, u)\rightarrow (\bar w, \bar u)}\dfrac{f_1(w, u)-f_1(\bar w,\bar u)- T_1(\bar w,\bar u)((w, u)-(\bar w,\bar u))}{\|(w, u)-(\bar w,\bar u)\|},$$ one has 
	\begin{align*}\begin{array}{lcl}
			L_1 &=	 & \displaystyle\lim_{(w, u)\rightarrow (\bar w,  \bar u)}\Big[
			\dfrac{(-x, -y, Ax - By-c)- (-\bar x, -\bar y, \bar A \bar x - \bar B \bar y - \bar c)}{\|(w-\bar w, u-\bar u)\|}\\
			& & \quad   -\dfrac{\big( -(x -\bar x), -(y -\bar y), \bar A (x-\bar x) - \bar B(y-\bar y) + (A -\bar A)\bar x -(B-\bar B)\bar y - (c-\bar c)\big)}{\|(w-\bar w, u-\bar u)\|}\Big]\nonumber\\
			&=	& \displaystyle\lim_{(w, u)\rightarrow (\bar w,  \bar u)}\bigg (0, 0, \dfrac{(A - \bar A)(x-\bar x) - (B -\bar B)(y-\bar y)}{\|(w-\bar w, u-\bar u)\|} \bigg)\\
			& = & 0,
		\end{array}
	\end{align*}  where the last equality holds true because
	\begin{align*}\begin{array}{lcl}\dfrac{\|(A - \bar A)(x-\bar x) - (B -\bar B)(y-\bar y)\|}{\|(w-\bar w, u-\bar u)\|} & \leq &  \dfrac{\|A - \bar A\|\|x-\bar x\| +\|B -\bar B\|\|y-\bar y\|}{\|(w-\bar w, u-\bar u)\|}\\ & \leq & \|A - \bar A\| + \|B -\bar B\|\end{array}
	\end{align*} and $ \|A - \bar A\| + \|B -\bar B\|\to 0$ as $(w, u)\rightarrow (\bar w,  \bar u)$. We have thus proved that $f_1$ is Fr\'echet differentiable  at $(\bar w, \bar u)$ and $\nabla f_1(\bar w, \bar u)=T_1(\bar w, \bar u)$. Since the function $$\mathcal T_1:W\times U\to L(W\times U, V)$$ assigning each vector $(\bar w, \bar u)\in W\times U$ to the linear operator $T_1(\bar w, \bar u)\in L(W\times U, V)$ is continuous on $W \times U$, we can assert that $f_1$ is strictly differentiable  at $(\bar w, \bar u)$ and its strict derivative is given by~\eqref{deri_f1}; see, e.g., \cite[p.~19]{B-M06}.
	
	Now, suppose that $\bar u \neq (0, 0)$. To verify the surjectivity of $\nabla f_1(\bar w, \bar u) : W \times U \to V$, take any $v=(v^1, v^2, v^3) \in \mathbb{R}^n \times \mathbb{R}^m \times \mathbb{R}^l = V$. We need to show that there exist $$w = (A, B, c) \in \mathbb{R}^{l\times n} \times \mathbb{R}^{l\times m} \times \mathbb{R}^l  =W\quad {\rm and}\quad  u =(x, y) \in \mathbb{R}^n \times \mathbb{R}^m = U$$ with the property
	\begin{equation}\label{eq1}
		\nabla f_1(\bar w, \bar u)(w, u) = v.
	\end{equation}
	From~\eqref{deri_f1} one has
	\begin{align*}
		\nabla f_1(\bar w, \bar u)(w, u) = v & \; \Longleftrightarrow \; (-x, -y, \bar A x - \bar B y + A \bar x - B \bar y - c) = (v^1, v^2, v^3)\\
		& \; \Longleftrightarrow \; \begin{cases}
			x = -v^1\\
			y = -v^2\\
			\bar A x - \bar B y + A \bar x - B \bar y-c = v^3.
		\end{cases}
	\end{align*}
	Thus, in order to have~\eqref{eq1}, one must choose $u=(x, y)=(-v^1, -v^2)$ and $w=(A, B, c)$ so that
	\begin{equation}\label{identity_6}
		A \bar x - B \bar y - c= v^3 + \bar A v^1 - \bar B v^2 .
	\end{equation}
	Put $\widetilde v^3 = v^3 + \bar A v^1 - \bar B v^2$ and assume that $\widetilde v^3=(\widetilde v^3_1, \widetilde v^3_2, \dots, \widetilde v^3_l)$. Then~\eqref{identity_6} is rewritten as 
	\begin{equation*}\label{identity_6a}
		A \bar x - B \bar y - c= \widetilde v^3.
	\end{equation*}
	The property $\bar u\neq (0, 0)$ implies that either $\bar x\neq 0$ or $\bar y\neq 0$. It suffices to consider the case $\bar x\neq 0$, because the  situation $\bar y\neq 0$ can be treated similarly. Suppose that $\bar x= (\bar x_1, \bar x_2, \dots, \bar x_n)$ with $\bar x_{\tilde j} \neq 0$ for some index $\tilde j \in \{1, \dots,  n\}$. Choose $B = 0 \in \mathbb{R}^{l\times m}$, $c =0 \in \mathbb R^l$ and define a matrix $A = (a_{ij}) \in \mathbb{R}^{l\times n}$ by setting $a_{i\tilde j} = \dfrac{\widetilde v^3_i}{\bar x_{\tilde j}}$ for every $i \in \{1, \dots, l\}$ and $a_{ij} = 0$ for every $i \in \{1, \dots, l\},\ j \in \{1, \dots, n\} \setminus \{\tilde j\}$. Then we have
	\begin{align*}
		A \bar x - B \bar y - c= A \bar x = \left (\begin{array}{c} a_{11}\bar x_1 
			+ a_{12} \bar x_2 + \cdots + a_{1n} \bar x_n \\
			a_{21}\bar x_1 + a_{22} \bar x_2 + \cdots + a_{2n} \bar x_n\\
			\vdots \\
			a_{l1}\bar x_1 + a_{l2} \bar x_2 + \cdots + a_{ln} \bar x_n
		\end{array}\right) = \left(\begin{array}{c}\widetilde v^3_1\\
			\widetilde v^3_2\\
			\vdots\\
			\widetilde v^3_l \end{array}\right).
	\end{align*}
	This means that the chosen element $w=(A, B, c)$ satisfies~\eqref{identity_6}. We have thus proved the surjectivity of the linear operator $\nabla f_1(\bar w, \bar u)$. The injectivity of $\nabla f_1(\bar w, \bar u)^*$ follows from the last fact and~\cite[Lemma 1.18]{B-M06}.
\end{proof}

\medskip
Second, we will obtain a formula for the limiting coderivative of~$G_1$ at any point in its graph.

\begin{lemma}\label{lemma_G1}
	The set-valued map $G_1: W\times U \rightrightarrows V$ in~\eqref{G1} has closed graph. For  any point $(\bar w, \bar u, \bar v) \in \gph G$, the limiting coderivative $D^*G_1(\bar w, \bar u, \bar v) : V \rightrightarrows W \times U$ is given by
	\begin{align}\label{coderi_G1}
		D^*G_1(\bar w, \bar u, \bar v)(v')  = \begin{cases}
			(0, 0), \quad &\mbox{if } - v' \in N\left(\bar v; \mathcal C \times \mathcal Q \times \{0_{\mathbb{R}^l}\} \right)\\
			\emptyset, \quad &\mbox{otherwise,}
		\end{cases}
	\end{align}
	where $v'\in V$ is arbitrarily chosen. 
\end{lemma}
\begin{proof}
	The closedness of $\mathcal C$ and $\mathcal Q$ together with the formula $\gph G_1 = W \times U \times (\mathcal C \times \mathcal Q \times \{0_{\mathbb{R}^l}\})$ shows that $G_1$ has closed graph. For any $(\bar w, \bar u, \bar v) \in \gph G_1$ and $v'\in V$, one has
	\begin{align*}
		& D^*G_1(\bar w, \bar u, \bar v)(v')\\ 
		&  = \{(w', u') \in W \times U  :  (w', u', -v')\in N \big((\bar w, \bar u, \bar v); \gph G_1\big)\}\\
		& = \{(w', u') \in W \times U  :  (w', u', -v')\in N \big((\bar w, \bar u, \bar v); W\! \times\! U \times (\mathcal C \! \times\! \mathcal Q \times\! \{0_{\mathbb{R}^l}\})\big)\}\\
		& = \{(w', u') \in W \times U  :  (w', u', -v')\in \{0_W\} \times \{0_U\} \times N \big( \bar v; \mathcal C \times \mathcal Q \times \{0_{\mathbb{R}^l}\}\big)\}.
	\end{align*}
	Thus, formula~\eqref{coderi_G1} is valid.
\end{proof}

Next, let us prove Theorem~\ref{Lipschitz-like_phi_thm}.

\medskip
\textbf{Proof of Theorem~\ref{Lipschitz-like_phi_thm}.}\ \, Suppose that $(\bar A, \bar B, \bar c) \in \mathbb{R}^{l\times n} \times \mathbb{R}^{l\times m} \times \mathbb R^l$, $(\bar x, \bar y) \in \Phi (\bar A,\bar B, \bar c)$, $\mathcal C$ is normally regular at $\bar x$, and $\mathcal Q$ is normally regular at $\bar y$.  

Before verifying the two assertions of the theorem, we need to make some preparations.

Put $\bar w=(\bar A, \bar B, \bar c)$, $\bar u = (\bar x, \bar y)$ and $\bar v:= -f_1(\bar w, \bar u)=(\bar x, \bar y, - \bar A \bar x + \bar B \bar y+\bar c).$ Thanks to formula~\eqref{phi_solution map}, we can study the Lipschitz-likeness of $\Phi$ at $(\bar w, \bar u) \in \gph \Phi$ by using Theorem~\ref{Thm. 4.2-Mor-2004-JOGO}, which deals with the Lipschitz-likeness of the solution map $S$ defined by~\eqref{solution_map} of the generalized equation~\eqref{para_general_eq}. Observe that $f_1$ is strictly differentiable at $(\bar w, \bar u)$ by Lemma~\ref{lemma_f1} and  $G_1$ has closed graph by Lemma~\ref{lemma_G1}. Moreover, since $\gph G_1 = W \times U \times (\mathcal C \times \mathcal Q \times \{0_{\mathbb{R}^l}\})$, $\gph G_1$ is normally regular at $(\bar w, \bar u, \bar v)$. Thus, $G_1$ is graphically regular at $(\bar w, \bar u, \bar v)$. In our setting, condition~\eqref{Lipschitz-like condition} reads as follows:
\begin{equation*}
	\big[(w', 0) \in \nabla f_1(\bar w, \bar u)^*(v')+D^*G_1(\bar w, \bar u, \bar v)(v') \big]\Longrightarrow [w'=0,\, v'=0].
\end{equation*}
According to~\eqref{coderi_G1}, this implication means the following: If $w'=(A', B', c') \in W$ and $v'=(x', y', z') \in V$ are such that 
\begin{equation}\label{identityphi_1a}
	- v' \in N\left(\bar v; \mathcal C \times \mathcal Q\times \{0_{\mathbb{R}^l}\} \right)
\end{equation}
and 
\begin{align}\label{identityphi_2}
	(w', 0) = \nabla f_1(\bar w, \bar u)^*(v'),
\end{align}
then one must have $w' = 0$ and $v'=0$. 

Since $N\left(\bar v; \mathcal C \times \mathcal Q\times \{0_{\mathbb{R}^l}\} \right) = N(\bar x, \mathcal C) \times N(\bar y, \mathcal Q) \times \mathbb{R}^l,$ the inclusion~\eqref{identityphi_1a} holds if and only if
\begin{align}\label{identityphi_1b}
	x' \in -N(\bar x, \mathcal C)\quad {\rm and}\quad  y' \in -N(\bar y, \mathcal Q).
\end{align}  Clearly, condition~\eqref{identityphi_2} can be rewritten as 
\begin{align*}
	\langle w', w \rangle =  \langle \nabla f_1(\bar w, \bar u)^*(v'), (w, u) \rangle, \quad \forall w = (A, B, c) \in W, \, u =(x, y) \in U
\end{align*}
or, equivalently,
\begin{align*}
	\langle w', w \rangle = \langle v', \nabla f_1(\bar w, \bar u) (w, u) \rangle,\quad \forall w = (A, B, c) \in W, \, u =(x, y) \in U.
\end{align*}
Therefore, by~\eqref{deri_f1} we can infer that~\eqref{identityphi_2} is satisfied if and only if 
\begin{align}\label{identityphi_2a}
	\langle w', w \rangle = -\langle x', x\rangle - \langle y', y \rangle+ \langle z', \bar A x - \bar B y + A \bar x - B \bar y - c\rangle 
\end{align}
for all $ w = (A, B, c) \in W$ and $u =(x, y) \in U$. We have thus shown that, in our setting, condition~\eqref{Lipschitz-like condition} means the following:
\begin{equation}\label{Lipschitz-like conditionphi}
	\left[w'=(A', B', c')\ {\rm and}\  v'=(x', y', z')\  {\rm satisfy}\  \eqref{identityphi_1b}\ {\rm and}\  \eqref{identityphi_2a}\right] \Longrightarrow \left[w' = 0, \,  v'=0\right].
\end{equation} 

(\textit{Sufficiency}) Suppose that condition~\eqref{Lipschitz-like_S_con} is fulfilled. To prove that the map $\Phi$ is Lipschitz-like at $(\bar w, \bar u)$ by the first assertion of Theorem~\ref{Thm. 4.2-Mor-2004-JOGO}, it suffices to verify that~\eqref{Lipschitz-like conditionphi} holds. 

Fix any $w'=(A', B', c')$ and $v'=(x', y', z')$ satisfying~\eqref{identityphi_1b} and~\eqref{identityphi_2a}. Applying equality~\eqref{identityphi_2a} with $w=(A, B, c) = (0, 0, 0)$ and $u=(x, y) = (x, 0)$ yields 
\begin{equation}\label{identityphi_2b}
	\langle x', x\rangle = \langle z', \bar A x\rangle,\quad \forall x\in \mathbb{R}^n.
\end{equation}
It follows that $x' = \bar A^{\rm T}z'$. So, by the first inclusion in~\eqref{identityphi_1b} we have $\bar A^{\rm T}z' \in -N(\bar x, \mathcal C)$, i.e.,
\begin{equation}\label{identityphi_3}
	z' \in \big(\bar A^{\rm T}\big)^{-1}\big(-N(\bar x; \mathcal C)\big).
\end{equation} Analogously, substituting $w=(A, B, c)=(0,0, 0)\in W$ and $u=(x, y)=(0, y)\in U$ to~\eqref{identityphi_2a} gives
\begin{equation}\label{identityphi_2c}
	\langle y', y\rangle = -\langle z', \bar B y\rangle, \quad \forall y\in \mathbb{R}^m.
\end{equation}
This is equivalent to
$y' =-\bar B^{\rm T}z'$.  Hence, one has $\bar B^{\rm T}z' \in N(\bar y, \mathcal Q)$ by the second inclusion in~\eqref{identityphi_1b}. Therefore, 
\begin{equation}\label{identityphi_4}
	z' \in \big(\bar B^{\rm T}\big)^{-1}\big(N(\bar y; \mathcal Q)\big).
\end{equation}
Thanks to~\eqref{Lipschitz-like_S_con}, from~\eqref{identityphi_3} and~\eqref{identityphi_4} we get $z'=0$. Consequently,  using~\eqref{identityphi_2b} and~\eqref{identityphi_2c}, we obtain $x'=0$ and $y' = 0$. Now, since $x' = 0$, $y'=0$, and $z'=0$, by~\eqref{identityphi_2a} we can infer that $w'=0$. We have shown that the implication~\eqref{Lipschitz-like conditionphi} is fulfilled. As a consequence, $\Phi$ is Lipschitz-like at $(\bar w, \bar u)$.

(\textit{Necessity}) Suppose that $\Phi$ is Lipschitz-like at $(\bar w, \bar u)=\big((\bar A, \bar B, \bar c),(\bar x, \bar y)\big)$ \textit{and the additional assumption $\bar u\neq (0, 0)$, where $\bar u=(\bar x, \bar y)$, is satisfied}. To apply the second assertion of Theorem~\ref{Thm. 4.2-Mor-2004-JOGO}, we need to verify condition~\eqref{regularity condition}.  In our setting, the latter is formulated as follows:
\begin{equation*}\label{regularityphi}
	\big[(0, 0) \in \nabla f_1(\bar w, \bar u)^*(v')+D^*G_1(\bar w, \bar u, \bar v)(v') \big]\Longrightarrow [v'=0].
\end{equation*}

Let $v'=(x', y', z') \in V$ be such that $(0, 0) \in \nabla f_1(\bar w, \bar u)^*(v')+D^*G_1(\bar w, \bar u, \bar v)(v').$ Because of~\eqref{coderi_G1}, the last inclusion means that $- v' \in N\left(\bar v; \mathcal C \times \mathcal Q\times \{0_{\mathbb{R}^l}\} \right)$ and
\begin{equation}\label{regularityphi_2}
	(0, 0) = \nabla f_1(\bar w, \bar u)^*(v').
\end{equation} 
Since  $\bar u \neq 0$,  by Lemma~\ref{lemma_f1} we know that $\nabla f_1(\bar w, \bar u)^* : V \to W \times U$ is an injective linear operator. So,~\eqref{regularityphi_2} yields $v'=(0, 0, 0)$. We have thus shown that condition~\eqref{regularity condition} is satisfied. Therefore, as $\Phi $ is Lipschitz-like at $(\bar w, \bar u)$, by the second assertion of Theorem~\ref{Thm. 4.2-Mor-2004-JOGO} we obtain~\eqref{Lipschitz-like condition}. In other words,~\eqref{Lipschitz-like conditionphi} is valid thanks to the preparations given before the proof of ``Sufficiency". We will use~\eqref{Lipschitz-like conditionphi} to prove that the regularity condition~\eqref{Lipschitz-like_S_con} holds.

Take any vector $z'$ with
\begin{align}\label{z'_phi} z'\in \big(\bar A^{\rm T}\big)^{-1}\big(-N(\bar x; \mathcal C)\big) \cap \big(\bar B^{\rm T}\big)^{-1}\big(N(\bar y; \mathcal Q)\big).
\end{align}
Then, there exist $x' \in -N(\bar x; \mathcal C)$ and $y' \in - N(\bar y; \mathcal Q)$ such that $\bar A^{\rm T}z'=x'$ and $\bar B^{\rm T}z'=-y'$. It follows that $\langle x', x\rangle = \langle z', \bar A x\rangle$ for all $x\in \mathbb{R}^n$ and $\langle y', y\rangle = -\langle z', \bar B y\rangle$ for all $y\in \mathbb{R}^m$. Consequently, one has
\begin{align}\label{equality0_phi}
	0= -\langle x', x\rangle - \langle y', y\rangle + \langle z', \bar A x - \bar B y\rangle,\ \; \forall x\in \mathbb{R}^n,\;  \forall y\in \mathbb{R}^m.
\end{align}

To proceed, let us show that there exists $w'=(A',B', c')\in W$ such that $w'$ together with the chosen vectors $x', y', z'$ satisfy~\eqref{identityphi_2a} for all $ w = (A, B, c) \in W$ and $u =(x, y) \in U$. Suppose that $\bar x = (\bar x_1, \bar x_2, \dots, \bar x_n)$, $\bar y = (\bar y_1, \bar y_2, \dots, \bar y_m)$, and $z' = (z'_1, z'_2, \dots, z'_l)$. Choose $w'=(A',B', c')$ with 
\begin{equation*}\label{A'-B'_phi}
	A':=\left (\begin{array}{c c c c} 
		z'_1\bar x_1 & z'_1 \bar x_2 &\cdots &z'_1\bar x_n\\
		z'_2\bar x_1 & z'_2 \bar x_2 &\cdots &z'_2\bar x_n\\
		\vdots & \vdots &\cdots &\vdots\\
		z'_l\bar x_1 & z'_l \bar x_2 &\cdots &z'_l\bar x_n
	\end{array}\right),\quad
	B':=\left (\begin{array}{c c c c} 
		-z'_1\bar y_1 & -z'_1 \bar y_2 &\cdots &-z'_1\bar y_m\\
		-z'_2\bar y_1 & -z'_2 \bar y_2 &\cdots &-z'_2\bar y_m\\
		\vdots & \vdots &\cdots &\vdots\\
		-z'_l\bar y_1 & -z'_l \bar y_2 &\cdots &-z'_l\bar y_m
	\end{array}\right),
\end{equation*} and $c':=-z'$. Fix any $w = (A, B, c) \in W$ with
\begin{equation*}
	A=\left (\begin{array}{c c c c} 
		a_{11} & a_{12} &\cdots &a_{1n}\\
		a_{21} & a_{22} &\cdots &a_{2n}\\
		\vdots & \vdots &\cdots &\vdots\\
		a_{l1}& a_{l2} &\cdots &a_{ln}
	\end{array}\right),\quad
	B=\left (\begin{array}{c c c c} 
		b_{11} & b_{12} &\cdots &b_{1m}\\
		b_{21} & b_{22} &\cdots &b_{2m}\\
		\vdots & \vdots &\cdots &\vdots\\
		b_{l1}& b_{l2} &\cdots &b_{lm}
	\end{array}\right).
\end{equation*}
Then, we have $$\langle A', A\rangle = \displaystyle \sum_{i=1}^l\sum_{j=1}^{n} (z'_i a_{ij} \bar x_j) = \displaystyle \sum_{i=1}^l z'_i \left(\displaystyle \sum_{j=1}^n a_{ij} \bar x_j\right) = \langle z', A \bar x \rangle,$$
$$\langle B', B\rangle = \displaystyle \sum_{i=1}^l \sum_{j=1}^m (-z'_i b_{ij} \bar y_j) = - \displaystyle \sum_{i=1}^l z'_i \left(\displaystyle \sum_{j=1}^{m} b_{ij} \bar y_j\right) = - \langle z', B \bar y \rangle,$$ and $\langle c', c \rangle = - \langle z', c \rangle$. It follows that $$\langle w' , w \rangle =\langle A', A\rangle+ \langle B', B\rangle + \langle c', c\rangle = \langle z', A \bar x - B \bar x - c\rangle$$ for all $w = (A, B, c) \in W$. This and~\eqref{equality0_phi} imply that $w'=(A', B', c')$ and $(x', y', z')$ fulfill the equality~\eqref{identityphi_2a} for all $ w = (A, B, c) \in W$ and $u =(x, y) \in U$. Thus, keeping in mind that $x'$ and $y'$ satisfy~\eqref{identityphi_1b} by their constructions, we deduce from~\eqref{Lipschitz-like conditionphi} where  $v':=(x', y', z')$ that $w' = 0$ and $(x', y', z') = (0, 0, 0)$. In particular, we have $z' =0$. As the element $z'$ satisfying~\eqref{z'_phi} was chosen arbitrarily, we have thus shown that~\eqref{Lipschitz-like_S_con} is valid.

The proof is complete. $\hfill\Box$

\medskip
Since a convex set is normally regular at any point belonging it, the next result follows directly from Theorem~\ref{Lipschitz-like_phi_thm}.

\begin{theorem}
	Suppose that $\mathcal C$ and $\mathcal Q$ are convex sets, $(\bar A, \bar B, \bar c) \in \mathbb{R}^{l\times n} \times \mathbb{R}^{l\times m} \times \mathbb R^l$ and $(\bar x, \bar y) \in \Phi (\bar A,\bar B, \bar c)$. If~\eqref{Lipschitz-like_S_con} holds, then the solution map $\Phi$ of {\rm (NSEP)} is Lipschitz-like at $\big((\bar A, \bar B, \bar c),(\bar x, \bar y)\big)$. Conversely, if $\Phi$ is Lipschitz-like at $\big((\bar A, \bar B, \bar c),(\bar x, \bar y)\big)$  and $(\bar x, \bar y) \neq (0, 0)$, then~\eqref{Lipschitz-like_S_con} is fulfilled.
\end{theorem}

\section{Nonhomogeneous Split Feasibility Problems}\label{sect_NSFP} \setcounter{equation}{0}
Let $\mathcal C\subset \mathbb{R}^n$, $\mathcal Q\subset \mathbb{R}^m$ be nonempty closed sets and let there be given a matrix $A\in \mathbb{R}^{m\times n}$ and a vector $b \in \mathbb R^m$. The problem of finding an $x \in \mathcal C$ such that $Ax +b \in \mathcal Q$ is called a \textit{nonhomogeneous split feasibility problem} (NSFP). 

Note that (NSFP) covers, for instance, the problems of finding solutions to 

- a linear equation (where $\mathcal C = \mathbb R^n$ and $\mathcal Q = \{0\}$, 

- the linear constraint system considered in~\cite{Huyen_Yen16} (where $\mathcal C = \mathbb R^n$), 

- the set-constraint linear system as discussed in~\cite[Appendix A]{Rusz06} (where $\mathcal C$ is convex and $\mathcal Q = \{0\}$). 

The classical split feasibility problem (SFP) introduced by Censor and Elfving~\cite{Censor_Elfving_94} is a particular case of (NSFP) with $\mathcal C,\, \mathcal Q$ being convex and $b = 0$. Recently, Chen et al.~\cite{Chen_etal_20} considered (SFP) in a nonconvex setting. The authors proposed a difference-of-convex approach to solve (SFP) and pointed out some applications to matrix factorizations and outlier detection. 

This section is devoted to the solution stability of (NSFP) when the data $A, b$ undergo small perturbations. The solution map $\Psi : \mathbb{R}^{m\times n} \times \mathbb R^m \rightrightarrows \mathbb{R}^{n}$  of (NSFP) is defined by 
\begin{equation}\label{psi}
	\Psi (A, b):= \big\{x \in \mathcal C \, : \, Ax +b \in \mathcal Q\big\}, \quad (A, b) \in \mathbb{R}^{m\times n} \times \mathbb R^m.
\end{equation}

The Lipschitz-likeness of $\Psi$ at a reference point in its graph can be characterized as follows.

\begin{theorem}\label{Lipschitz-like_psi_thm}
	Let $(\bar A, \bar b) \in \mathbb{R}^{m\times n} \times \mathbb R^m$ be given and let $\bar x \in \Psi(\bar A, \bar b)$. Suppose that $\mathcal C$ is normally regular at $\bar x$ and $\mathcal Q$ is normally regular at $\bar A \bar x + \bar b$. If the regularity condition
	\begin{equation}\label{Lipschitz-like_psi_con}
		\big(\bar A^{\rm T}\big)^{-1}(-N(\bar x; \mathcal C)) \cap N(\bar A \bar x+\bar b; \mathcal Q)=\{0\}
	\end{equation}
	holds, then the solution map $\Psi $ of {\rm (NSFP)} is Lipschitz-like at $\big((\bar A, \bar b), \bar x\big)$. Conversely, if $\Psi $ is Lipschitz-like at $\big((\bar A, \bar b), \bar x\big)$ and if $\bar x \neq 0$, then~\eqref{Lipschitz-like_psi_con} is valid.
\end{theorem}

Similarly as it was done the preceding section, to prove Theorem~\ref{Lipschitz-like_psi_thm} we will transform the nonhomogeneous split feasibility problem in question to a generalized equation and apply Theorem~\ref{Thm. 4.2-Mor-2004-JOGO}. 

\medskip
Put  $W = \mathbb{R}^{m\times n} \times \mathbb R^m$. Consider the function $f_2: W\times \mathbb{R}^n \to \mathbb{R}^n \times \mathbb{R}^m$ given by   
\begin{equation}\label{f2}
	f_2(w, x):= (-x, -Ax -b), \quad w=(A, b) \in  W, \ x \in \mathbb{R}^n
\end{equation}
and the set-valued map $G_2: W\times \mathbb{R}^{m\times n} \rightrightarrows \mathbb{R}^n \times \mathbb R^m$ with 
\begin{equation}\label{G2}
	G_2(w, x):= \mathcal C \times \mathcal Q, \quad w=(A, b) \in  W, \ x \in \mathbb{R}^n.
\end{equation}
Then, the solution map~\eqref{psi} can be treated as the solution map $w\mapsto \Psi(w)$ of the parametric generalized equation 
$0 \in f_2(w, x) + G_2(w, x)$. Namely, one has
\begin{equation}\label{psi_solution map}
	\Psi(w) = \left\{x \in \mathbb{R}^n \; : \; 0 \in f_2(w, x) + G_2(w, x)\right\}, \quad w \in W.
\end{equation}

Some basic properties of $f_2$ and $G_2$ are shown in the following two lemmas.

\begin{lemma}\label{lemma_f2}
	The function $f_2: W \times \mathbb R^n \to \mathbb R^n \times \mathbb R^m$ defined by~\eqref{f2} is strictly differentiable at any $(\bar w, \bar x) \in W \times \mathbb R^n$ with $\bar w = (\bar A, \bar b)$. The derivative $\nabla f_2(\bar w, \bar x): W \times \mathbb R^n \to \mathbb R^n \times \mathbb R^m$ of $f_2$ at $(\bar w, \bar x)$ is given by
	\begin{equation}\label{deri_f2}
		\nabla f_2(\bar w, \bar x)(w) = (-x, - \bar A x  - A \bar x -b), \quad w= (A, b) \in W, \ x \in \mathbb R^n.
	\end{equation}
	Moreover, if $\bar x \neq 0$, then  the operator $\nabla f_2(\bar w, \bar x) : W \times \mathbb R^n \to \mathbb R^n \times \mathbb R^m$ is surjective and thus its adjoint operator  $\nabla f_2(\bar w, \bar x)^* : \mathbb R^n \times \mathbb R^m \to W \times \mathbb R^n$ is injective.
\end{lemma}
\begin{proof}
	Fix any element  $(\bar w, \bar x) \in W \times \mathbb R^n$ with $\bar w = (\bar A, \bar b)$ and consider the linear operator $T_2(\bar w, \bar x): W \times \mathbb R^n \to \mathbb R^n \times \mathbb R^m$, 
	\begin{equation*}\label{T2}
		T_2(\bar w, \bar x) (w, x):= (-x, - \bar A x  - A \bar x- b), \quad w =(A, b) \in W,\  x \in\mathbb R^n.
	\end{equation*}
	Note that
	\begin{align*}\begin{array}{rcl}
			L_2:=	& &\displaystyle\lim_{(w, x)\rightarrow (\bar w,  \bar x)}\dfrac{f_2(w, x)-f_2(\bar w,  \bar x)- T_2(\bar w,  \bar x)((w, x)-(\bar w,  \bar x))}{\|(w, x)-(\bar w,  \bar x)\|}\nonumber\\
			=	& & \displaystyle\lim_{(w, x)\rightarrow (\bar w,  \bar x)}\Big[
			\dfrac{(-x, -Ax-b)- (-\bar x, -\bar A \bar x - \bar b)}{\|A-\bar A\|+\|B-\bar B\|+\| x-\bar x\|}  \\ & & \qquad  \qquad \; \; \; - \dfrac{\big(-(x -\bar x), - \bar A (x-\bar x) - (A -\bar A)\bar x - (b-\bar b)\big)}{\|A-\bar A\|+\|B-\bar B\|+\| x-\bar x\|} \Big]\nonumber\\
			=	& & \displaystyle\lim_{(w, x)\rightarrow (\bar w,  \bar x)}\bigg (0,  \dfrac{-(A - \bar A)(x-\bar x)}{\|A-\bar A\|+\|B-\bar B\|+\|x-\bar x\|} \bigg).
		\end{array}
	\end{align*}
	It follows that $L_2 = 0$. Hence, $f_2$ is Fr\'echet differentiable  at $(\bar w, \bar x)$ and $\nabla f_2(\bar w, \bar x)=T_2(\bar w, \bar x)$. Since the map $\mathcal T_2:W \times \mathbb R^n   \to L(W \times \mathbb R^n , \mathbb R^n \times \mathbb R^m)$ putting each $(\bar w, \bar x)\in W\times \mathbb R^n$ in correspondence with the linear operator $T_2(\bar w, \bar x)\in L(W \times \mathbb R^n , \mathbb R^n \times \mathbb R^m)$ is continuous on $W \times \mathbb R^n$, we can infer that $f_2$ is strictly differentiable  at $(\bar w, \bar x)$ and its strict derivative is given by~\eqref{deri_f2}; see, e.g., \cite[p.~19]{B-M06}.  
	
	To justify the second assertion of the theorem, suppose that $\bar x = (\bar x_1, \bar x_2, \dots, \bar x_n)$ with $\bar x_{\tilde j} \neq 0$ for some $\tilde j \in \{1, 2, \dots, n\}$. The surjectivity of $\nabla f_2(\bar w, \bar x): W \times \mathbb R^n \to \mathbb R^n \times \mathbb R^m$ means that for any $(u, v) \in \mathbb R^n \times \mathbb R^m$ there exist $w= (A, b) \in W$ and  $x \in \mathbb R^n$ satisfying 
	\begin{equation}\label{eq}
		\nabla f_2(\bar w, \bar x)(w, x) = (u, v).
	\end{equation}
	Using~\eqref{deri_f2} we have $\nabla f_2(\bar w, \bar x)(w, x) = (u, v)$ if and only if $(-x, - \bar A x  - A \bar x -b) = (u, v)$. Thus, the relation~\eqref{eq} holds when $x = -u$ and the components $A, b$ of $w$ satisfy
	\begin{align}\label{identitypsi_5}
		A \bar x + b= \bar A u - v.
	\end{align}
	Set $\widetilde v =  \bar A u - v$ and suppose that $ \widetilde v=(\widetilde v_1, \widetilde v_2, \dots, \widetilde v_m) \in \mathbb R^m$. Then, \eqref{identitypsi_5} is equivalent to 
	\begin{equation}\label{eq2}
		A \bar x + b = \widetilde v .
	\end{equation}
	Define $b=0_{\mathbb R^m}$, $A = (a_{ij}) \in \mathbb{R}^{m\times n}$ with $a_{i\tilde j} = \dfrac{\widetilde v_i}{\bar x_{\tilde j}}$ for $i \in \{1, \dots,  m\}$ and $a_{ij} = 0$ for $i \in \{1, \dots, m\}, j \in \{1, \dots, n\} \setminus \{\tilde j\}$. Then, we have
	\begin{align*}
		A \bar x + b = A \bar x = \left (\begin{array}{c} a_{11}\bar x_1 
			+ a_{12} \bar x_2 + \cdots + a_{1n} \bar x_n \\
			a_{21}\bar x_1 + a_{22} \bar x_2 + \cdots + a_{2n} \bar x_n\\
			\vdots\\
			a_{m1}\bar x_1 + a_{m2} \bar x_2 + \cdots + a_{mn} \bar x_n
		\end{array}\right) = \left(\begin{array}{c}\widetilde v_1\\
			\widetilde v_2\\
			\vdots\\
			\widetilde v_m \end{array}\right).
	\end{align*}
	Thus, the chosen pair $(A, b)$ satisfies~\eqref{eq2}. We have proved that $\nabla f_2(\bar w, \bar x)$ is surjective. The injectivity of the adjoint operator $\nabla f_2(\bar w, \bar x)^*$ follows from~\cite[Lemma 1.18]{B-M06}. 
	
	The proof is complete.
\end{proof}

\begin{lemma}\label{lemma_G2}
	The set-valued map $G_2: W \times \mathbb R^n \rightrightarrows \mathbb R^n \times \mathbb R^m$ defined by~\eqref{G2} has closed graph. For any given point $\big((\bar w, \bar x), (\bar u, \bar v)\big) \in \gph G_2$ with $\bar w= (\bar A, \bar b)$, the limiting coderivative $D^*G_2\big((\bar w, \bar x), (\bar u, \bar v)\big): \mathbb R^n \times \mathbb R^m \rightrightarrows W \times \mathbb R^n$ is given by
	\begin{equation}\label{coderi_G2}
		D^*G_2\big((\bar w, \bar x), (\bar u, \bar v)\big)(u', v')=
		\begin{cases}
			\{(0, 0)\}, \quad &\mbox{if } u' \in -N(\bar u; \mathcal C), \ v' \in -N(\bar v; \mathcal Q)\\
			\emptyset, & \mbox{otherwise},
		\end{cases}
	\end{equation}
	where $(u', v')$ is an arbitrary point in $\mathbb R^n \times \mathbb R^m$.
\end{lemma}

The proof of this lemma is omitted because it is similar to that of  Lemma~\ref{lemma_G1}.

\medskip
\textbf{Proof of Theorem~\ref{Lipschitz-like_psi_thm}.}\ \, 
Fix any $(\bar A, \bar b) \in \mathbb{R}^{m\times n} \times \mathbb R^m$ and take $\bar x \in \Psi(\bar A, \bar b)$. Put  $\bar w = (\bar A, \bar b)$ and $(\bar u, \bar v) = -f_2(\bar w, \bar x) = (\bar x, \bar A \bar x+\bar b).$   The representation~\eqref{psi_solution map} allows us to apply Theorem~\ref{Thm. 4.2-Mor-2004-JOGO} to study the Lipschitz-likeness of the solution map $\Psi $. By Lemmas~\ref{lemma_f1} and~\ref{lemma_G1}, $f_2$ is strictly differentiable at $(\bar w, \bar x)$ and $G_1$ has closed graph. Besides, as  $$\gph G_2= W \times \mathbb R^n \times \mathcal C \times \mathcal Q,$$ where $\mathcal C$ is normally regular at $\bar x$ and $\mathcal Q$ is normally regular at $\bar A \bar x + \bar b$ by the assumptions of the theorem, the set $\gph G_2$ is normally regular at the point $\big((\bar w, \bar x), (\bar u, \bar v)\big)$ belonging to it. Thus, $G_2$ is graphically regular at $\big((\bar w, \bar x), (\bar u, \bar v)\big)$.  

To go furthermore, we need to explore condition~\eqref{Lipschitz-like condition}. In our setting, the latter means 
\begin{equation*}
	\big[(w', 0) \in \nabla f_2(\bar w, \bar x)^*(u', v') + D^*G_2\big((\bar w, \bar x), (\bar u, \bar v)\big) (u', v') \big]\Longrightarrow [w' = 0, \, u'= 0, \, v'=0].
\end{equation*}
By~\eqref{coderi_G2}, this implication means that if $w' = (A', b') \in \mathbb R^{m \times n} \times \mathbb R^m $ and $(u', v') \in \mathbb R^n \times \mathbb R^m$ are such that 
\begin{equation}\label{inclusionpsi_1}
	u' \in - N(\bar x; \mathcal C),\ \, v' \in - N(\bar A \bar x+ \bar b; \mathcal Q)
\end{equation}
and 
\begin{equation}\label{inclusionpsi_2}
	(w', 0) = \nabla f_2(\bar w, \bar x)^*(u', v'),
\end{equation}
then one must have $w' = (A', b') = (0, 0)$ and $(u', v') = (0, 0)$. Condition~\eqref{inclusionpsi_2} can be restated equivalently as 
\begin{equation*}
	\langle w', w \rangle= \langle \nabla f_2(\bar A, \bar x)^*(u', v'), (w, x) \rangle, \quad \forall w= (A, b) \in W, \; \forall x \in \mathbb R^n,
\end{equation*}
which means that
\begin{equation*}
	\langle w', w \rangle  = \langle (u', v'), \nabla f_2(\bar A, \bar x) (w, x) \rangle, \quad  \forall w= (A, b) \in W, \; \forall x \in \mathbb R^n.
\end{equation*}
Therefore, by~\eqref{deri_f2} we can rewrite~\eqref{inclusionpsi_2} as 
\begin{equation}\label{inclusionpsi_2a}
	\langle w', w \rangle  = - \langle  u', x \rangle -  \langle v',  \bar A x + A \bar x + b\rangle, \quad  \forall w= (A, b) \in W, \; \forall x \in \mathbb R^n.
\end{equation}
Summing up, we can equivalently restate condition~\eqref{Lipschitz-like condition} as
\begin{equation}\label{Lipschitz-like conditionpsi}
	\left[w'=(A', b')\ {\rm and}\  (u', v')\  {\rm satisfy}\  \eqref{inclusionpsi_1}\ {\rm and}\  \eqref{inclusionpsi_2a}\right] \Longrightarrow \left[w' = 0, \, u' = 0, \, v'=0\right].
\end{equation} 

Now, we can prove the two assertions of the theorem.

(\textit{Sufficiency}) Suppose that condition~\eqref{Lipschitz-like_psi_con} is fulfilled. If we can show that~\eqref{Lipschitz-like conditionpsi} holds, then~$\Psi $ is Lipschitz-like at $(\bar w, \bar x)$ by the first assertion of Theorem~\ref{Thm. 4.2-Mor-2004-JOGO}. Fix any $w'=(A', b')$ and $(u', v')$ satisfying~\eqref{inclusionpsi_1} and~\eqref{inclusionpsi_2a}. Applying~\eqref{inclusionpsi_2a} with $w=(A, b) = (0, 0) \in W$, and $x \in \mathbb R^n$, we get
\begin{equation*}
	\langle  u', x \rangle = -\langle v',  \bar A x \rangle, \quad \forall x \in \mathbb R^n.
\end{equation*}
This implies that $u' = \bar A^{\rm T}(- v')$. Therefore, by the first inclusion in~\eqref{inclusionpsi_1}, we obtain $-v' \in \big(\bar A^{\rm T}\big)^{-1}(-N(\bar x; \mathcal C))$. Combining this with the second inclusion in~\eqref{inclusionpsi_1} yields $$-v' \in \big(\bar A^{\rm T}\big)^{-1}(-N(\bar x; \mathcal C)) \cap N(\bar A \bar x+\bar b; \mathcal Q).$$ So, by the assumption~\eqref{Lipschitz-like_psi_con} we have $-v' = 0$. Hence, $v' = 0$ and $u' = \bar A^{\rm T} (-v') = 0$. Now, substituting $(u', v')=(0, 0)$ into~\eqref{inclusionpsi_2a} yields $w' = (A', b') =(0, 0)$. We have shown that~\eqref{Lipschitz-like conditionpsi} holds and, thus, justified the first assertion of the theorem.

(\textit{Necessity}) Suppose that $\Psi $ is Lipschitz-like at $(\bar w, \bar x)$ with $\bar w = (\bar A, \bar b)$ and $\bar x \neq 0$. The second assertion of Theorem~\ref{Thm. 4.2-Mor-2004-JOGO} can be used, if condition~\eqref{regularity condition} is fulfilled.  In our setting, the latter means that
\begin{equation}\label{regularitypsi}
	\big[(0, 0) \in \nabla f_2(\bar w, \bar x)^*(u', v') + D^*G_2\big((\bar w, \bar x), (\bar u, \bar v)\big) (u', v') \big]\Longrightarrow [(u', v') = (0, 0)].
\end{equation}
Let $(u', v') \in \mathbb R^n \times \mathbb R^m$ be such that
\begin{equation}\label{inclusionpsi}
	(0, 0) \in \nabla f_2(\bar w, \bar x)^*(u', v') + D^*G_2\big((\bar w, \bar x), (\bar u, \bar v)\big) (u', v').
\end{equation}
From~\eqref{coderi_G2} and~\eqref{inclusionpsi} it follows that $u' \in - N(\bar x; \mathcal C)$, $v' \in -N(\bar A \bar x+\bar b; \mathcal Q)$, and
\begin{equation}\label{casespsi}
	\nabla f_2(\bar w, \bar x)^*(u', v') = (0, 0).
\end{equation}
Since $\bar x \neq 0$, Lemma~\ref{lemma_f2} assures that $\nabla f_2(\bar w, \bar x)^* : \mathbb R^n \times \mathbb R^m \to W \times \mathbb R^n$ is injective. Thus,~\eqref{casespsi} implies that $(u', v') = (0, 0)$. So,~\eqref{regularitypsi} holds. Therefore, recalling that $\Psi $ is Lipschitz-like at $(\bar w, \bar x)$, by Theorem~\ref{Thm. 4.2-Mor-2004-JOGO} we get~\eqref{Lipschitz-like condition}. In other words, we have~\eqref{Lipschitz-like conditionpsi}. We will derive~\eqref{Lipschitz-like_psi_con} from~\eqref{Lipschitz-like conditionpsi}.

Take any $v'\in \big(\bar A^{\rm T}\big)^{-1}\big(-N(\bar x; \mathcal C)\big) \cap N\big(\bar A \bar x+\bar b; \mathcal Q\big)$. Then, $-v' \in -N\big(\bar A \bar x+\bar b; \mathcal Q\big)$ and there is $u' \in -N(\bar x; \mathcal C)$ such that $\bar A^{\rm T}(-v') = -u'$. The last equality implies that
\begin{equation}\label{inclusionpsi_3}
	0 = - \langle  u', x \rangle - \langle -v',  \bar A x \rangle, \quad \forall x \in \mathbb R^n.
\end{equation}
There exists $w' = (A', b') \in W$ such that~\eqref{inclusionpsi_2a}, where $-v'$ takes the place of $v'$, is satisfied. Indeed, suppose that $\bar x = (\bar x_1, \bar x_2, \dots, \bar x_n)$ and $v' = (v'_1, v'_2, \dots, v'_m)$. Choose $b'= v' \in \mathbb R^m$ and 
\begin{equation*}\label{A'_psi}
	A'=\left (\begin{array}{c c c c} 
		v'_1\bar x_1 & v'_1 \bar x_2 &\cdots &v'_1\bar x_n\\
		v'_2\bar x_1 & v'_2 \bar x_2 &\cdots &v'_2\bar x_n\\
		\vdots & \vdots &\cdots &\vdots\\
		v'_m\bar x_1 & v'_m \bar x_2 &\cdots &v'_m\bar x_n
	\end{array}\right)  \in \mathbb{R}^{m\times n}.
\end{equation*}
For any $w = (A, b) \in W$ with 
\begin{equation*}
	A=\left (\begin{array}{c c c c} 
		a_{11} & a_{12} &\cdots &a_{1n}\\
		a_{21} & a_{22} &\cdots &a_{2n}\\
		\vdots & \vdots &\cdots &\vdots\\
		a_{m1}& a_{m2} &\cdots &a_{mn}
	\end{array}\right),
\end{equation*}
we have $\langle b', b \rangle = \langle v', b \rangle$ and $$\langle A', A\rangle = \displaystyle \sum_{i=1}^m\sum_{j=1}^{n} (v'_i a_{ij} \bar x_j)= \displaystyle \sum_{i=1}^m v'_i \left(\displaystyle \sum_{j=1}^n a_{ij} \bar x_j\right) = \langle v', A \bar x \rangle.$$ Thus, $\langle w', w \rangle =\langle v',  A \bar x + b\rangle$ for all $w = (A, b) \in W$.
Combining this with~\eqref{inclusionpsi_3}, we get
\begin{equation*}
	\langle w', w \rangle =  - \langle  u', x \rangle -  \langle - v',  \bar A x + A \bar x + b\rangle
\end{equation*}
for any $w = (A, b) \in W$ and $x \in \mathbb R^n$. This means that~\eqref{inclusionpsi_2a}, where $-v'$ takes the place of $v'$, holds. Since $u' \in -N(\bar x; \mathcal C)$ and $-v' \in -N\big(\bar A \bar x; \mathcal Q\big)$, we see that~\eqref{inclusionpsi_1}, where $-v'$ takes the place of $v'$, is valid. Therefore, by~\eqref{Lipschitz-like conditionpsi} we have $w'=0$, $u'=0$, and $-v'=0$. In particular, $v'=0$. Because $v'\in \big(\bar A^{\rm T}\big)^{-1}\big(-N(\bar x; \mathcal C)\big) \cap N\big(\bar A \bar x+\bar b; \mathcal Q\big)$ was taken arbitrarily, we obtain~\eqref{Lipschitz-like_psi_con}.

The proof is complete. $\hfill\Box$

\medskip
Since a convex set is normally regular at any point belonging to it, the next statement immediately follows from Theorem~\ref{Lipschitz-like_psi_thm}.

\begin{theorem}
	Suppose that $\mathcal C$ and $\mathcal Q$ are convex sets. Let $(\bar A, \bar b) \in \mathbb{R}^{m\times n} \times \mathbb R^m$ and $\bar x \in \Psi(\bar A, \bar b)$. If condition~\eqref{Lipschitz-like_psi_con} is satisfied, then the solution map $\Psi $ of {\rm (NSFP)} is Lipschitz-like at $\big((\bar A, \bar b), \bar x\big)$. Conversely, if $\Psi $ is Lipschitz-like at $\big((\bar A, \bar b), \bar x\big)$  and $\bar x \neq 0$, then it is necessary that the condition~\eqref{Lipschitz-like_psi_con} is fulfilled.
\end{theorem}

\section{Illustrative Examples}\label{sect_Examples} \setcounter{equation}{0}

For all the examples in this section, we choose $n=2,\ m=1,\ l=1$, and $\mathcal Q=\mathbb R_+$. Assumptions and assertions of Theorems~\ref{Lipschitz-like_phi_thm} and~\ref{Lipschitz-like_psi_thm} will be verified in detail.

The next example illustrates the applicability of Theorem~\ref{Lipschitz-like_psi_thm} for nonhomogeneous split feasibility problems with possibly nonconvex constraint sets.

\begin{ex}\label{ex_1} {\rm Let $\mathcal C:=\big\{x=(x_1, x_2)\in \mathbb{R}^2\, :\, x_1\leq x_2^2 \big\}$. With $A=\big(\begin{matrix}
			a_{11} & a_{12}
		\end{matrix}\big)\in\mathbb R^{1\times 2}$ and $b\in\mathbb R$, the solution map in~\eqref{psi} of the nonhomogeneous split feasibility problem becomes\ \, 
		$\Psi (A, b)= \big\{x\in \mathcal C \, : \, a_{11}x_1 +a_{12}x_2+b\geq 0\big\}.$ The analysis given in Example~\ref{ex_a} shows that $\mathcal C$ is nonconvex but normally regular at any point from the set. For $\bar A:=\big(\begin{matrix}
			1 & -2
		\end{matrix}\big)$ and $\bar b:=1$, one has
		$$\begin{array}{rcl}
			\Psi (\bar A,\bar b)& = & \big\{x=(x_1,x_2)\in\mathbb R^2\, : \,  x_1\leq x_2^2,\ x_1 -2x_2+1\geq 0\big\}\\
			& = & \big\{x=(x_1,x_2)\in\mathbb R^2\, : \, 2x_2-1 \leq x_1\leq x_2^2\big\}. \end{array}$$ For $x_2=1$, the condition $2x_2-1 \leq x_1\leq x_2^2$ in the above description of the solution set $\Psi (\bar A,\bar b)$ implies that $x_1=1$. Meanwhile, for every $x_2>1$ (resp., for every $x_2<1$), this condition forces $x_1\in [2x_2-1,x_2^2]\subset\mathbb R$, where $2x_2-1<x_2^2$. Thus, the intersection of $\Psi (\bar A,\bar b)$ with the horizontal straight line $x_2=\beta$ in $\mathbb R^2$ is a line segment having distinct endpoints, provided that $\beta\neq 1$. Note also that $\Psi (\bar A,\bar b)=\Omega_1\cup \Omega_2$, where $\Omega_1:=\{x \, : \, 2x_2-1 \leq x_1\leq x_2^2, x_2\leq 1\}$ and $\Omega_2=\{x \, : \, 2x_2-1 \leq x_1\leq x_2^2,\ x_2\geq 1\}$ (see Fig.~\ref{fig1}). One has $\Omega_1\cap \Omega_2=\{(1,1)\}$ and each set $\Omega_i$ for $i\in\{1,2\}$ is a connected unbounded domain in $\mathbb R^2$. 
		\begin{figure}[ht!]
			\centering
			\includegraphics[width=0.8\linewidth]{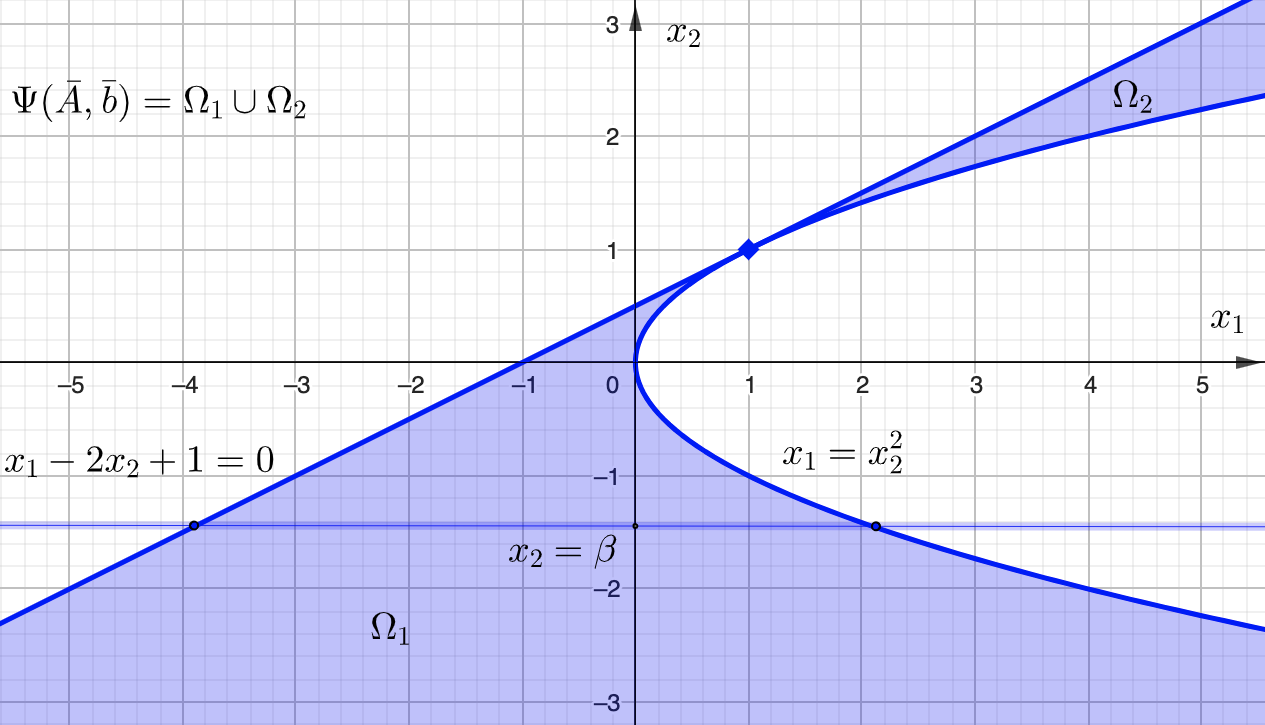}
			\caption{The solution set $\Psi(\bar A, \bar b)$ of (NSFP) in Example~\ref{ex_1}}
			\label{fig1}
		\end{figure}
		
		Take any point $\bar x=(\bar x_1,\bar x_2)\in \Psi (\bar A,\bar b)\setminus\{(1,1)\}$. If $\bar x_1<\bar x_2^2$, then $\bar x\in {\rm int}\,\mathcal C$; so $N(\bar x; \mathcal C)=\{(0,0)\}$. It follows that $$\big(\bar A^{\rm T}\big)^{-1}(-N(\bar x; \mathcal C))=\left\{\lambda\in\mathbb R\,:\, \left(\begin{matrix}
			\lambda\\ -2\lambda
		\end{matrix}\right)=\left(\begin{matrix}
			0\\ 0
		\end{matrix}\right)\right\}=\{0\}.$$ Hence condition~\eqref{Lipschitz-like_psi_con} is satisfied. So, the solution map $\Psi $ of {\rm (NSFP)} is Lipschitz-like at $\big((\bar A, \bar b), \bar x\big)$ by the first assertion of Theorem~\ref{Lipschitz-like_psi_thm}. If $\bar x_1=\bar x_2^2$, then $\bar x$ is a boundary point of $\mathcal C$. Using formula~(1.18) from~\cite{B-M06}, we can verify that $N(\bar x; \mathcal C)=\{(t,-2\bar x_2t)\, :\, t\geq 0\}$. Therefore, one has $$\big(\bar A^{\rm T}\big)^{-1}(-N(\bar x; \mathcal C))=\left\{\lambda\in\mathbb R\,:\, \left(\begin{matrix}
			\lambda\\ -2\lambda
		\end{matrix}\right)=\left(\begin{matrix}
			-t\\ 2\bar x_2t
		\end{matrix}\right) \ \,{\rm for\ some}\ \, t\geq 0\right\}=\{0\}.$$ Hence condition~\eqref{Lipschitz-like_psi_con} is again satisfied, and Theorem~\ref{Lipschitz-like_psi_thm} assures that $\Psi$ is Lipschitz-like at $\big((\bar A, \bar b), \bar x\big)$. Finally, consider the solution $\bar x=(1,1)\in\Psi (\bar A,\bar b)$. Since $\bar A \bar x+\bar b=0$, one has $N(\bar A \bar x+\bar b; \mathcal Q)=-\mathbb R_+$.  Noting that $\bar x=(1,1)$ is a boundary point of of $\mathcal C$, by~\cite[formula~(1.18)]{B-M06} we get $N(\bar x; \mathcal C)=\{(t,-2t)\, :\, t\geq 0\}$. It follows that $$\big(\bar A^{\rm T}\big)^{-1}(-N(\bar x; \mathcal C))=-\mathbb R_+.$$ Thus, condition~\eqref{Lipschitz-like_psi_con} is violated. As  $\bar x\neq 0$, by the second assertion of Theorem~\ref{Lipschitz-like_psi_thm} we can infer that $\Psi$ is \textit{not} Lipschitz-like at $\big((\bar A, \bar b), \bar x\big)$.}
\end{ex}

\begin{ex}\label{ex_2} {\rm Define $\mathcal C=\big\{x=(x_1, x_2)\in \mathbb{R}^2\, :\, 2\leq x_1^2+x_2^2 \leq 5\big\}$ and consider the nonhomogeneous split equality problem defined by the data tube $(A, B, c, \mathcal C, \mathcal Q)$ with $A\in\mathbb R^{1\times 2}$, $B\in\mathbb R^{1\times 1}$ and $c\in\mathbb R$ , which aims at finding all the pairs $(x, y) \in \mathcal C \times \mathcal Q$ such that $Ax - By = c$. As it has been shown in Example~\ref{ex_b}, the nonconvex and compact set $\mathcal C$ is normally regular at any point of it. Since $\mathcal Q=\mathbb R_+$ is a convex set, the normal regularity is available at any point belonging to it.  Let us choose $\bar A=\big(\begin{matrix}
			1 & 1
		\end{matrix}\big)\in\mathbb R^{1\times 2}$, $\bar B=(\frac{1}{2})\in\mathbb R^{1\times 1}$, and $\bar c=1$. 
		\begin{figure}[ht!]
			\begin{minipage}[c][1\width]{
					0.48\textwidth}
				\centering
				\includegraphics[width=1\textwidth]{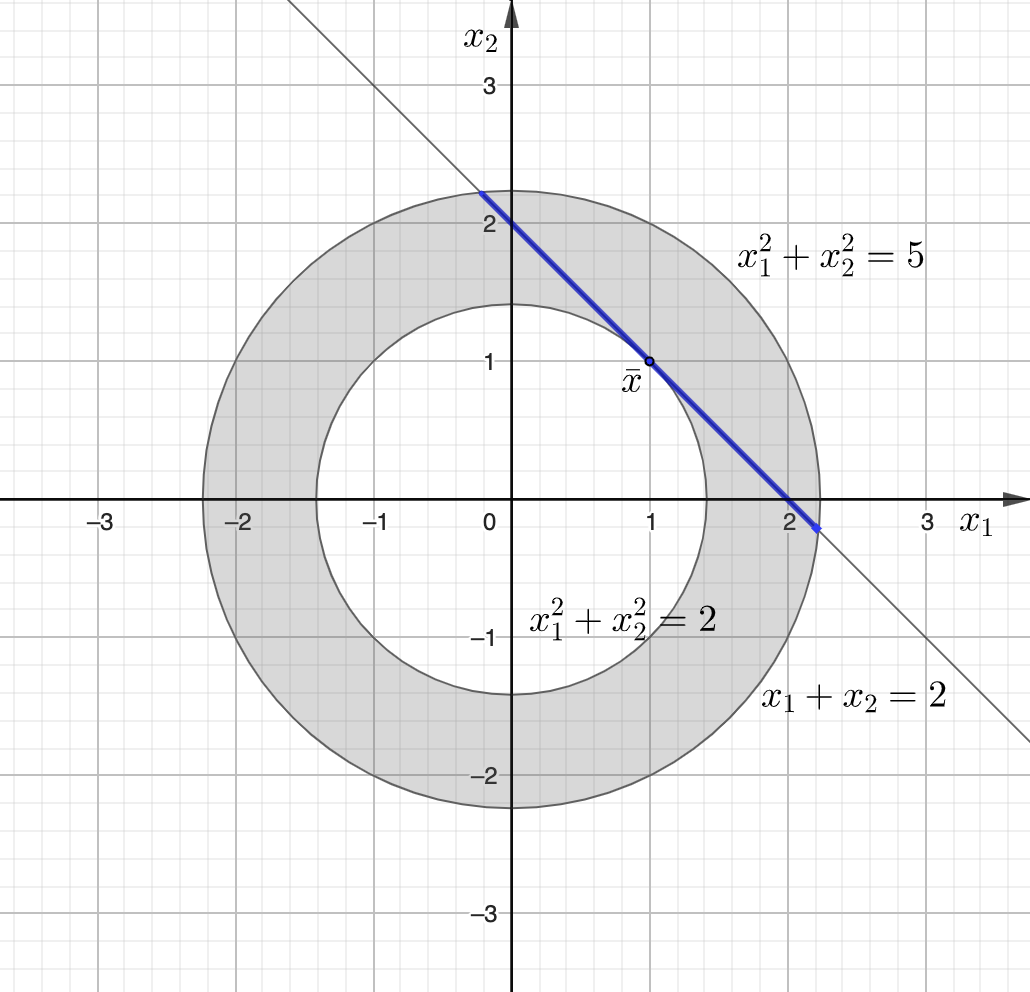}
			\end{minipage}
			\hfill 	
			\begin{minipage}[c][1\width]{
					0.48\textwidth}
				\centering
				\includegraphics[width=1\textwidth]{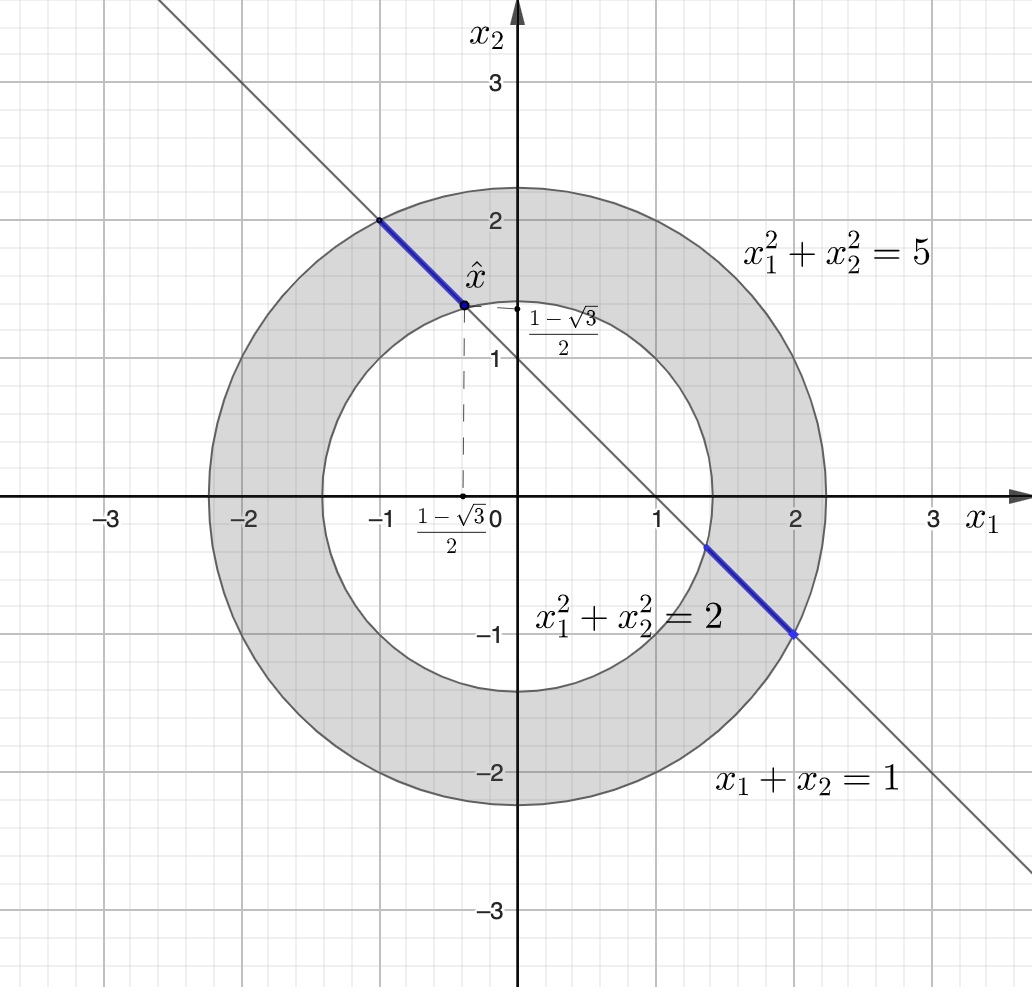}
			\end{minipage}
			\caption{Projections of the solution set $\Phi(\bar A, \bar B, \bar c)$ of (NSEP) in Example~\ref{ex_2} on the plane $y = 2$ (left) and on the plane $y = 0$ (right)}
		\end{figure}
		First, take $\bar x=(1,1)$ and $\bar y=2$. Then 
		$(\bar x, \bar y) \in \Phi (\bar A,\bar B, \bar c)$. The regularity condition
		\begin{equation*}\label{Lipschitz-like_S_con}
			\big(\bar A^{\rm T}\big)^{-1}\big(-N(\bar x; \mathcal C)\big) \cap \big(\bar B^{\rm T}\big)^{-1}\big(N(\bar y; \mathcal Q)\big)  = \{0\}
		\end{equation*} is satisfied, because $N(\bar y; \mathcal Q)=\{0\}$. So, the solution map $\Phi$ of {\rm (NSEP)} is Lipschitz-like at $\big((\bar A, \bar B, \bar c),(\bar x, \bar y)\big)$ by Theorem~\ref{Lipschitz-like_phi_thm}.
		
		Next, define $\hat x=\left(\dfrac{1-\sqrt{3}}{2},\dfrac{1+\sqrt{3}}{2}\right)$ and $\hat y=0$. Then 
		$(\hat x, \hat y) \in \Phi (\bar A,\bar B, \bar c)$. The regularity condition
		\begin{equation*}\label{Lipschitz-like_S_con}
			\big(\bar A^{\rm T}\big)^{-1}\big(-N(\hat x; \mathcal C)\big) \cap \big(\bar B^{\rm T}\big)^{-1}\big(N(\hat y; \mathcal Q)\big)  = \{0\}
		\end{equation*} is fulfilled. Indeed, here we have $N(\bar y; \mathcal Q)=-\mathbb R_+$, but $\big(\bar A^{\rm T}\big)^{-1}\big(-N(\hat x; \mathcal C)\big)=\{0\}$. Thus, the solution map $\Phi$ of {\rm (NSEP)} is Lipschitz-like at $\big((\bar A, \bar B, \bar c),(\hat x, \hat y)\big)$ by Theorem~\ref{Lipschitz-like_phi_thm}.}
\end{ex}

\section{Conclusions}\label{sect_Conclusions} \setcounter{equation}{0}

By introducing a right-hand-side perturbation to the constraint system of the classical split equality problem and the split feasibility problem, we get a \textit{nonhomogeneous split equality problem} (resp.,~\textit{nonhomogeneous split feasibility problem}). These problems have been transformed into suitable parametric generalized equations and used some tools of generalized differentiation and a fundamental result of Mordukhovich~\cite{Mor04}. This approach allows us to characterize the Lipschitz-like property of the solution maps in question. Because of the appearance of the canonical perturbation of the problem, the results of this paper are independent of that   in~\cite{HuongXuYen22}, even when the constraint sets are assumed to be convex.

Two examples have been designed to analyze the obtained necessary and sufficient conditions for the Lipschitz-likeness of the solution maps. 

\section*{Statements and Declarations}
\textbf{Acknowledgements}:  This work has been co-funded by the European Union (European Regional Development Fund EFRE, fund number STIIV-001), the National Natural Science Foundation of China (grant number U1811461), the Australian Research Council/Discovery Project (grant number DP200100124), and the Vietnam Academy of Science and Technology (project code NCXS02.01/24-25). 

\medskip
\noindent\textbf{Author Contributions}: All authors contributed equally, read and approved the final manuscript.

\medskip
\noindent\textbf{Competing Interests}: The authors have no relevant financial or non-financial interests to disclose.


\begin{thebibliography}{}
	
	\bibitem{Aubin84}  Aubin JP. {Lipschitz behavior of solutions to convex minimization problems}. Math. Oper. Res. 1984;{9}:87--111.
	
	\bibitem{Censor_Elfving_94} Censor Y, Elfving T. {A multiprojection algorithm using Bregman projections in product space}. Numer. Algorithms  1994;{8}:221--239.
	
	\bibitem{Chen_etal_20} Chen C, Pong TK, Tan L, Zeng L. {A difference-of-convex approach for split feasibility with applications to matrix factorizations and outlier detection}. J. Global Optim.  2020;{78}:107--136.
	
	\bibitem{Huyen_Yen16} Huyen DTK,  Yen ND. {Coderivatives and the solution map of a linear constraint system}. SIAM J. Optim.  2016;{26}:986--1007.
	
	\bibitem{HuongXuYen22} Huong VT, Xu HK, Yen ND. {Stability analysis of split equality and split feasibility problems}. 2024; arXiv:2410.16856. (Submitted)
	
	\bibitem{Moudafi13} Moudafi A. {A relaxed alternating CQ-algorithms for convex feasibility problems}. Nonlinear Anal.  2013;{79}:117--121.
	
	\bibitem{Moudafi14} Moudafi A. {Alternating CQ-algorithms for convex feasibility and split fixed-point problems}. J. Nonlinear Convex Anal.  2014;{15}:809--818.
	
	\bibitem{Moudafi_Shemas13} Moudafi A, Al-Shemas  E. {Simultaneous iterative methods for split equality problems and applications}. Trans. Math. Program. Appl.  2013;{1}:1--11.
	
	\bibitem{Byrne_Moudafi_17} Byrne C, Moudafi A. {Extensions of the CQ algorithms for the split feasibility and split equality problems}. J. Nonlinear Convex Anal.   2017;{18}:1485--1496.
	
	\bibitem{B-M06}  Mordukhovich BS. {Variational Analysis and Generized Differentiation, Vol. I: Basic Theory}. Springer, Berlin, 2006.
	
	\bibitem{Mordukhovich_2018} Mordukhovich BS. {Variational Analysis and Applications}. Springer, Switzerland, 2018.
	
	\bibitem{Mor04} Mordukhovich BS. {Coderivative analysis of variational systems}. J. Global Optim.  2004;{28}:347--362.
	
	\bibitem{Reich_Tuyen_21} Reich S, Tuyen TM. {A new approach to solving split equality problems in Hilbert spaces}. Optimization  2022;{71}:4423--4445.
	
	\bibitem{Rusz06} Ruszcz\'ynski AP. {Nonlinear Optimization}. Princeton University Press, 2006.
	
	\bibitem{Xu_Celgieski21} Xu HK, Cegielski A. {The Landweber operator approach to the split equality problem}. SIAM J. Optim.  2021;{31}:626--652.
	
\end{thebibliography}
\end{document}